\documentclass[a4paper]{amsart}
\usepackage{graphicx}
\usepackage{amsmath}
\usepackage{amssymb}
\usepackage{oldgerm}
\usepackage[all]{xy}

\newcommand{\Hom}{\operatorname{Hom}\nolimits}

\renewcommand{\mod}{\operatorname{mod}\nolimits}

\newcommand{\gldim}{\operatorname{gl.dim}\nolimits}

\newcommand{\Ext}{\operatorname{Ext}\nolimits}

\newcommand{\HH}{\operatorname{HH}\nolimits}
\newcommand{\Ho}{\operatorname{H}\nolimits}

\newcommand{\rank}{\operatorname{rank}\nolimits}
\newcommand{\m}{\operatorname{\mathfrak{m}}\nolimits}

\newcommand{\ra}{\mathfrak{r}}
\newcommand{\La}{\Lambda}

\newcommand{\cx}{\operatorname{cx}\nolimits}
\newcommand{\op}{\operatorname{op}\nolimits}

\newcommand{\e}{\operatorname{e}\nolimits}
\newcommand{\Lae}{\Lambda^{\e}}

\newcommand{\repdim}{\operatorname{repdim}\nolimits}
\newcommand{\Krulldim}{\operatorname{Krull dim}\nolimits}
\newcommand{\End}{\operatorname{End}\nolimits}

\newtheorem{theorem}{Theorem}[section]

\newtheorem{corollary}[theorem]{Corollary}

\newtheorem{lemma}[theorem]{Lemma}
\newtheorem{proposition}[theorem]{Proposition}
\theoremstyle{definition}

\theoremstyle{definition}
\newtheorem*{question}{Question}
\theoremstyle{definition}

\theoremstyle{definition}

\theoremstyle{definition}
\newtheorem*{examples}{Examples}

\theoremstyle{definition}

\theoremstyle{remark}
\newtheorem*{remark}{Remark}
\theoremstyle{definition}

\theoremstyle{definition}
\newtheorem*{assumption}{Assumption}

\begin{document}
\title{Representation dimension and finitely generated cohomology}
\author{Petter Andreas Bergh}
\address{Petter Andreas Bergh \newline Institutt for matematiske fag \\
  NTNU \\ N-7491 Trondheim \\ Norway}
\email{bergh@math.ntnu.no}

%\date{\today}

\thanks{The author was supported by NFR Storforsk grant no.\
167130}



\maketitle

\begin{abstract}
We consider selfinjective Artin algebras whose cohomology groups are
finitely generated over a central ring of cohomology operators. For
such an algebra, we show that the representation dimension is
strictly greater than the maximal complexity occurring among its
modules. This provides a unified approach to computing lower bounds
for the representation dimension of group algebras, exterior
algebras and Artin complete intersections. We also obtain new
examples of classes of algebras with arbitrarily large
representation dimension.
\end{abstract}

\section{Introduction}\label{intro}

In his 1971 notes \cite{Auslander1}, Auslander introduced the
notion of the representation dimension of an Artin algebra. This
invariant measures how far an algebra is from having finite
representation type; it was introduced in order to study algebras
of infinite representation type. A non-semisimple algebra is of
finite type if and only if its representation dimension is exactly
two, and of infinite type if and only if the representation
dimension is at least three.

For a long time it was unclear whether there could exist algebras of
representation dimension strictly greater than three. Moreover,
Igusa and Todorov showed in \cite{Igusa} that if this was not the
case, i.e.\ if the representation dimension could not exceed three,
then the finitistic dimension conjecture would hold. However, in
2006 Rouquier showed in \cite{Rouquier2} that the representation
dimension of the exterior algebra on a $d$-dimensional vector space
is $d+1$, using the notion of the dimension of a triangulated
category (cf.\ \cite{Rouquier1}). Other examples illustrating this
were subsequently given in \cite{Avramov2}, \cite{BerghOppermann},
\cite{Krause}, \cite{Oppermann1} and \cite{Oppermann2}.

In this paper we study selfinjective Artin algebras satisfying a
certain finite generation hypothesis on its cohomology groups. We
show that the representation dimension of such an algebra is
strictly greater than the maximal complexity occurring among its
modules. This provides a unified approach to computing the known
lower bounds for the representation dimension of group algebras,
exterior algebras and Artin complete intersections. We also obtain
new examples of classes of algebras with arbitrarily large
representation dimension.

\section{Representation dimension}\label{repdim}

Throughout this paper, we let $k$ be a commutative Artin ring and
$\La$ an Artin $k$-algebra with Jacobson radical $\ra$. We denote by
$\mod \La$ the category of finitely generated $\La$-modules. The
\emph{representation dimension} of $\La$, denoted $\repdim \La$, is
defined as
$$\repdim \La \stackrel{\text{def}}{=} \inf \{ \gldim
\End_{\La}(M) \mid M \text{ generates and cogenerates } \mod \La
\},$$ where $\gldim$ denotes the global dimension of an algebra.
Auslander showed that the representation dimension of a
selfinjective algebra is at most its Loewy length, whereas Iyama
showed in \cite{Iyama} that this invariant is finite for every
Artin algebra.

In order to compute the representation dimension of exterior
algebras, Rouquier used the notion of the dimension of a
triangulated category, a concept he introduced in
\cite{Rouquier1}. We recall here the definitions. Let
$\mathcal{T}$ be a triangulated category, and let $\mathcal{C}$
and $\mathcal{D}$ be subcategories of $\mathcal{T}$. We denote by
$\langle \mathcal{C} \rangle$ the full subcategory of
$\mathcal{T}$ consisting of all the direct summands of finite
direct sums of shifts of objects in $\mathcal{C}$. Furthermore, we
denote by $\mathcal{C} \ast \mathcal{D}$ the full subcategory of
$\mathcal{T}$ consisting of objects $M$ such that there exists a
distinguished triangle
$$C \to M \to D \to C[1]$$
in $\mathcal{T}$, with $C \in \mathcal{C}$ and $D \in
\mathcal{D}$. Finally, we denote the subcategory $\langle
\mathcal{C} \ast \mathcal{D} \rangle$ by $\mathcal{C} \diamond
\mathcal{D}$. Now define $\langle \mathcal{C} \rangle_1$ to be
$\langle \mathcal{C} \rangle$, and for each $n \ge 2$ define
inductively $\langle \mathcal{C} \rangle_n$ to be $\langle
\mathcal{C} \rangle_{n-1} \diamond \langle \mathcal{C} \rangle$.
The \emph{dimension} of $\mathcal{T}$, denoted $\dim \mathcal{T}$,
is defined as
$$\dim \mathcal{T} \stackrel{\text{def}}{=} \inf \{ d \in
\mathbb{Z} \mid \text{ there exists an object } M \in \mathcal{T}
\text{ such that } \mathcal{T} = \langle M \rangle_{d+1} \}.$$ In
other words, the dimension of $\mathcal{T}$ is the minimal number of
layers needed to obtain $\mathcal{T}$ from one of its objects.

The key ingredient in the proof of our main result is the
following lemma on compositions of natural transformations. The
lemma is analogous to \cite[Lemma 4.11]{Rouquier1}.

\begin{lemma}\label{sequence}
Let $\mathcal{T}$ be a triangulated category, let $H_1, \dots,
H_{n+1}$ be cohomological functors on $\mathcal{T}$, and for each
$1 \le i \le n$ let $H_i \xrightarrow{f_i} H_{i+1}$ be a natural
transformation. Furthermore, let $\mathcal{C}_1, \dots,
\mathcal{C}_n$ be subcategories of $\mathcal{T}$ closed under
shifts, and assume that for every object $c \in \mathcal{C}_i$ the
map $H_i(c[j]) \xrightarrow{f_i} H_{i+1}(c[j])$ vanishes for $j
\gg 0$ (respectively, for $j \ll 0$). Then for every object $w \in
\mathcal{C}_1 \diamond \cdots \diamond \mathcal{C}_n$ the map
$H_1(w[j]) \xrightarrow{f_n \cdots f_1} H_{n+1}(w[j])$ vanishes
for $j \gg 0$ (respectively, for $j \ll 0$).
\end{lemma}

\begin{proof}
We may assume $n \ge 2$. Let $c_1 \to c \to c_2 \to c_1[1]$ be a
triangle in $\mathcal{T}$ with $c_1 \in \mathcal{C}_1$ and $c_2 \in
\mathcal{C}_2$. Then for every $j \in \mathbb{Z}$, there is a
commutative diagram
$$\xymatrix{
H_1(c_1[j]) \ar[r] \ar[d]^{f_1} & H_1(c[j]) \ar[r] \ar[d]^{f_1} &
H_1(c_2[j]) \ar[d]^{f_1} \\
H_2(c_1[j]) \ar[r] \ar[d]^{f_2} &
H_2(c[j]) \ar[r] \ar[d]^{f_2} & H_2(c_2[j]) \ar[d]^{f_2} \\
H_3(c_1[j]) \ar[r] & H_3(c[j]) \ar[r] & H_3(c_2[j]) }$$ with exact
rows. By assumption, there is an integer $j_1$ such that the
vertical upper left map vanishes for $j \ge j_1$, and an integer
$j_2$ such that the vertical lower right map vanishes for $j \ge
j_2$. An easy diagram chase shows that the vertical middle
composition vanishes for $j \ge \max \{ j_1, j_2 \}$, hence for
every object $w \in \mathcal{C}_1 \diamond \mathcal{C}_2$ the map
$H_1(w[j]) \xrightarrow{f_2 f_1} H_3(w[j])$ vanishes for $j \gg 0$.
An induction argument now establishes the lemma.
\end{proof}

The triangulated category we shall use is the \emph{stable module
category} of $\La$, in the case when $\La$ is selfinjective. This
category, denoted $\underline{\mod} \La$, is defined as follows: the
objects of $\underline{\mod} \La$ are the same as in $\mod \La$, but
two morphisms in $\mod \La$ are equal in $\underline{\mod} \La$ if
their difference factors through a projective $\La$-module. The
cosyzygy functor $\Omega_{\La}^{-1} \colon \underline{\mod} \La \to
\underline{\mod} \La$ is an equivalence of categories, and a
triangulation of $\underline{\mod} \La$ is given by using this
functor as a shift and by letting short exact sequences in $\mod
\La$ correspond to triangles. Thus $\underline{\mod} \La$ is a
triangulated category, and its dimension is related to the
representation dimension of $\La$ by the following result.

\begin{proposition}\cite[Proposition 3.7]{Rouquier2}\label{lowerbound}
If $\La$ is selfinjective and not semisimple, then $\repdim \Lambda
\ge \dim ( \underline{\mod} \Lambda ) +2$.
\end{proposition}

This result was originally formulated for a finite dimensional
algebra over a field, but it works just as well in our setting,
i.e.\ for an Artin algebra which is not necessarily finite
dimensional over a field.

The main result of this paper relates the representation dimension
of $\La$ with the maximal complexity occurring among its finitely
generated modules. Recall therefore that for a module $M \in \mod
\La$ with minimal projective resolution
$$\cdots \to P_2 \to P_1 \to P_0 \to M \to 0,$$
say, the \emph{complexity} of $M$ is defined as
$$\cx M \stackrel{\text{def}}{=} \inf \{ t \in \mathbb{N} \cup \{ 0
\} \mid \exists a \in \mathbb{R} \text{ such that } \ell_k (P_n)
\leq an^{t-1} \text{ for } n \gg 0 \}.$$ In general, the complexity
of a module may be infinite, whereas it is zero if and only if the
module is projective. The complexity of $M$ can be computed as the
rate of growth of the graded $k$-module $\Ext_{\La}^* (M, \La / \ra
)$, and from the definition we also see that it equals the
complexity of $\Omega_{\La}^i(M)$ for any $i \in \mathbb{N}$.
Moreover, given a short exact sequence
$$0 \to X_1 \to X_2 \to X_3 \to 0$$
in $\mod \La$, it is well known that the inequality $\cx X_u \le
\sup \{ \cx X_v, \cx X_w \}$ holds for $\{ u,v,w \} = \{ 1,2,3 \}$.
In particular, induction on the length of a module shows that $\cx X
\le \cx \La / \ra$ for every $X \in \mod \La$. We end this section
with the following elementary lemma, which shows that a module
generating $\underline{\mod} \La$ must be of maximal complexity.

\begin{lemma}\label{maxcx}
Let $\La$ be selfinjective, let $M \in \mod \La$ be a module, and
suppose there exists a number $n \in \mathbb{N}$ such that $\langle
M \rangle_n = \underline{\mod} \La$. Then $\cx N \le \cx M$ for
every $N \in \mod \La$, in particular $\cx M = \cx \La / \ra$.
\end{lemma}

\begin{proof}
The result follows from the fact that triangles in $\underline{\mod}
\La$ correspond to short exact sequences in $\mod \La$.
\end{proof}

\section{Finitely generated cohomology}\label{fingen}

We now introduce a certain ``finite generation" assumption on the
cohomology groups of $\La$. Recall that for $\La$-modules $X$ and
$Y$, the graded $k$-module $\Ext_{\La}^*(X,Y)$ is an
$\Ext_{\La}^*(Y,Y)-\Ext_{\La}^*(X,X)$-bimodule via Yoneda products.
Also recall that a graded $k$-module $\bigoplus V_i$ is of
\emph{finite type} if each $V_i$ is a finitely generated $k$-module.

\begin{assumption}[\bf{Fg}]
There exists a commutative Noetherian graded $k$-algebra $H =
\bigoplus_{i=0}^{\infty} H^i$ of finite type satisfying the
following:
\begin{enumerate}
\item[(i)] For every $M \in \mod \La$ there is a graded ring homomorphism
$$\phi_M \colon H \to \Ext_{\La}^*(M,M).$$
\item[(ii)] For each pair $(X,Y)$ of finitely generated
$\La$-modules, the scalar actions from $H$ on $\Ext_{\La}^*(X,Y)$
via $\phi_X$ and $\phi_Y$ coincide, and $\Ext_{\La}^*(X,Y)$ is a
finitely generated $H$-module.
\end{enumerate}
\end{assumption}

In the assumption, why do we require that the left and right scalar
multiplications on $\Ext_{\La}^*(X,Y)$ coincide? The reason is that
this requirement is what makes the bifunctor $\Ext_{\La}^* (-,-)$
preserve maps. To see this, let $f \colon M \to M'$ be a
homomorphism in $\mod \La$. For every $N \in \mod \La$, this map
induces a homomorphism $\hat{f} \colon \Ext_{\La}^*(M',N) \to
\Ext_{\La}^*(M,N)$ of graded groups. The image of a homogeneous
element
$$\theta \colon 0 \to N \to X_n \to \cdots \to X_1 \to
M' \to 0$$ is the extension $\theta f$ given by the
commutative diagram
$$\xymatrix{
\theta f \colon 0 \ar[r] & N \ar[r] \ar@{=}[d] & X_n \ar[r]
\ar@{=}[d] & \cdots \ar[r] & X_2 \ar[r] \ar@{=}[d] & Y \ar[r]
\ar[d] & M \ar[r] \ar[d]^f & 0 \\
\theta \colon 0 \ar[r] & N \ar[r] & X_n \ar[r] & \cdots \ar[r] &
X_2 \ar[r] & X_1 \ar[r] & M' \ar[r] & 0 }$$ in which the
module $Y$ is a pullback. For a homogeneous element $\eta \in
H$ we then get
\begin{eqnarray*}
\hat{f} ( \theta \cdot \eta ) & = & \hat{f} ( \eta \cdot \theta ) \\
& = & \hat{f} \left ( \phi_N ( \eta ) \circ \theta \right ) \\
& = & \phi_N ( \eta ) \circ ( \theta f) \\
& = & \eta \cdot \hat{f} ( \theta ) \\
& = & \hat{f} ( \theta ) \cdot \eta,
\end{eqnarray*}
showing $\hat{f}$ is a homomorphism of $H$-modules. Similarly,
$\Ext_{\La}^*(-,-)$ preserves maps in the second argument. The fact
that $\Ext_{\La}^*(-,-)$ preserves maps is absolutely essential.
Note that when using this property, induction on length of modules
shows that the finite generation part of the assumption {\bf{Fg}} is
equivalent to $\Ext_{\La}^*( \La / \ra, \La / \ra )$ being a
finitely generated $H$-module.

\sloppy It should also be noted that when {\bf{Fg}} holds, then
every finitely generated $\La$-module has finite complexity, i.e.\
$\cx \La / \ra < \infty$. Namely, the $H$-module $\Ext_{\La}^*( \La
/ \ra, \La / \ra )$ is finitely generated, and so its rate of growth
is not more than that of $H$. The ring $H$ is commutative Noetherian
and of finite type, and hence its rate of growth is finite.

In the following examples we point out three important situations
in which the assumption {\bf{Fg}} holds.

\begin{examples}
(i) Suppose $k$ is a field of positive characteristic $p$, and let
$G$ be a finite group whose order is divisible by $p$. Then by a
theorem of Evens (cf.\ \cite{Evens}), the graded commutative group
cohomology ring $\Ho^*(G,k) = \Ext_{kG}^* (k,k)$ is Noetherian.
Moreover, if $X_1$ and $X_2$ are finitely generated $kG$-modules,
then $\Ext_{kG}^*(X_1,X_2)$ is a finitely generated
$\Ho^*(G,k)$-module via the maps
$$ - \otimes_k X_i \colon \Ho^*(G,k) \to \Ext_{kG}^*(X_i,X_i),$$
and the right and left scalar actions induced by these maps commute
up to a graded sign. The even part $\bigoplus H^{2i} (G,k)$ of
$\Ho^*(G,k)$ is a commutative $k$-algebra, over which $\Ho^*(G,k)$
is finitely generated as a module.

(ii) Let $(A, \m, k)$ be a commutative Noetherian local complete
intersection. That is, the completion $\widehat{A}$, with respect to
the maximal ideal $\m$, is of the form $R / (x_1, \dots, x_c)$,
where $R$ is regular local and $x_1, \dots, x_c$ is a regular
sequence. We may without loss of generality assume that the length
$c$ of the defining regular sequence is the codimension of $A$,
i.e.\ $c = \dim_k \left ( \m / \m^2 \right ) - \dim A$. By
\cite[Section 1]{Avramov1} there exists a polynomial ring
$\widehat{A}[\chi_1, \dots, \chi_c]$ in commuting \emph{Eisenbud
operators}, such that for every finitely generated
$\widehat{A}$-module $X$ there is a homomorphism
$$\phi_X \colon \widehat{A}[\chi_1, \dots, \chi_c] \to
\Ext_{\widehat{A}}^*(X,X)$$ of graded rings. Moreover, for every
finitely generated $\widehat{A}$-module $Y$, the left and right
scalar actions on $\Ext_{\widehat{A}}^*(X,Y)$ coincide, and the
latter is a finitely generated $\widehat{A}[\chi_1, \dots,
\chi_c]$-module. Now if $A$ is Artin, then it is a complete ring
since $\m$ is nilpotent. Thus {\bf{Fg}} holds in this case.

(iii) Suppose $\La$ is projective as a $k$-module. Denote by $\Lae$
the enveloping algebra $\La \otimes_k \La^{\op}$ of $\La$, and by
$\HH^* (\La)$ its Hochschild cohomology ring. By \cite[Proposition
3]{Yoneda}, this is a graded commutative ring, and equal to
$\Ext_{\Lae}^* ( \La, \La )$ since $\La$ is $k$-projective. If $X_1$
and $X_2$ are finitely generated $\La$-modules, then the right and
left scalar actions from $\HH^* (\La)$ on $\Ext_{\La}^*(X_1,X_2)$,
via the maps
$$- \otimes_{\La} X_i \colon \HH^*(\La) \to \Ext_{\La}^*(X_i,X_i),$$
are graded commutative.

In \cite{Erdmann} the finite generation assumption imposed was the
following: there exists a commutative Noetherian graded subalgebra
$S = \bigoplus_{i=0}^{\infty} S^i$ of $\HH^* (\La)$, with $S^0 =
\HH^0 (\La)$, such that $\Ext_{\La}^*(\La / \ra, \La / \ra )$ is a
finitely generated $S$-module. However, having to deal with such an
``unknown" subalgebra of the Hochschild cohomology ring is not
satisfactory, and in fact it is not difficult to see (cf.\
\cite[Proposition 5.7]{Solberg}) that the assumption is equivalent
to the following one: the Hochschild cohomology ring $\HH^* (\La)$
is Noetherian and $\Ext_{\La}^*(\La / \ra, \La / \ra )$ is a
finitely generated $\HH^* (\La)$-module. Therefore, as in the first
example, we see that {\bf{Fg}} holds by choosing $H$ to be the even
part $\bigoplus \HH^{2i} ( \La )$ of $\HH^* (\La)$.

In particular, the assumption {\bf{Fg}} holds for exterior algebras.
Namely, suppose $k$ is a field, let $n$ be a number, and denote by
$\La$ the algebra
$$k \langle x_1, \dots, x_n \rangle / ( x_i^2, x_ix_j + x_jx_i ),$$
i.e.\ $\La$ is the exterior algebra on an $n$-dimensional vector
space. Then by \cite[Theorem 9.2 and Theorem 9.11]{Solberg}, the
Koszul dual $\Ext_{\La}^*(k,k)$ of $\La$ is the polynomial ring
$k[x_1, \dots, x_n]$, and via the map
$$- \otimes_{\La} k \colon \HH^*(\La) \to \Ext_{\La}^* (k,k)$$
this is a finitely generated $\HH^* (\La)$-module. By choosing $H$
to be the even part of the inverse image of $k[x_1, \dots, x_n]$, we
see that {\bf{Fg}} holds.
\end{examples}

Having pointed out these three examples where {\bf{Fg}} holds, we
now prove the main result: when $\La$ is selfinjective and {\bf{Fg}}
holds, then the dimension of the stable module category of $\La$ is
at least $\cx \La / \ra -1$.

\begin{theorem}\label{dimension}
If $\La$ is selfinjective and {\bf{Fg}} holds, then $\dim (
\underline{\mod} \Lambda ) \ge \cx \La / \ra -1$.
\end{theorem}

\begin{proof}
Denote the complexity of $\La / \ra$ by $c$. Let $M \in \mod \La$ be
a module generating $\underline{\mod} \Lambda$, i.e.\ there exists a
number $n$ such that $\langle M \rangle_n = \underline{\mod}
\Lambda$. Our aim is to show that $n \ge c$. If $c \le 1$, then
there is nothing to prove, so we may assume $c \ge 2$.

By Lemma \ref{maxcx} the module $M$ must have maximal complexity,
that is, the equality $\cx M = c$ holds. Choose, by
\cite[Proposition 2.1]{Bergh1}, a homogeneous element $\eta_1 \in
H^+$ such that the multiplication map
$$\Ext_{\La}^i(M,M \oplus \La / \ra ) \xrightarrow{\cdot \eta_1}
\Ext_{\La}^{i+|\eta_1|}(M,M \oplus \La / \ra )$$ is injective for $i
\gg 0$. Applying the map $\phi_M$ to $\eta_1$ gives a short exact
sequence
$$\phi_M ( \eta_1 ) \colon 0 \to M \xrightarrow{f_1} K_1 \to
\Omega_{\La}^{|\eta_1|-1}(M) \to 0,$$ and the arguments used in the
proof of \cite[Theorem 3.2]{Bergh1} shows that $\cx K_1 = c-1$.
Next, if $c \ge 3$, choose a homogeneous element $\eta_2 \in H^+$
such that the multiplication map
$$\Ext_{\La}^i(K_1,M \oplus \La / \ra ) \oplus \Ext_{\La}^i(M,K_1)
\xrightarrow{\cdot \eta_2} \Ext_{\La}^{i+|\eta_2|}(K_1,M \oplus \La
/ \ra ) \oplus \Ext_{\La}^{i+|\eta_2|}(M,K_1)$$ is injective for $i
\gg 0$. Applying the map $\phi_{K_1}$ to $\eta_2$ gives a short
exact sequence
$$\phi_{K_1} ( \eta_2 ) \colon 0 \to K_1 \xrightarrow{f_2} K_2 \to
\Omega_{\La}^{|\eta_2|-1}(K_1) \to 0,$$ in which $\cx K_2 = c -2$.
We continue this process until we end up with a module $K_{c-1}$ of
complexity $1$. Thus we obtain homogeneous elements $\eta_1, \dots,
\eta_{c-1} \in H^+$, and for each $1 \le j \le c-1$ a short exact
sequence
$$\phi_{K_{j-1}} ( \eta_j ) \colon 0 \to K_{j-1} \xrightarrow{f_j} K_j \to
\Omega_{\La}^{|\eta_j|-1}(K_{j-1}) \to 0$$ with $\cx K_j = c-j$
(here $K_0 =M$). For each $j$ the element $\eta_j$ is chosen in such
a way that it is regular on $\Ext_{\La}^i(K_{j-1},M \oplus \La / \ra
) \oplus \Ext_{\La}^i(M,K_{j-1})$ for $i \gg 0$.

For each $1 \le j \le c-1$ and $i \gg 0$, the exact sequence
$\phi_{K_{j-1}}( \eta_j )$ induces the two exact sequences
$$\xymatrix@C=15pt{
\Ext_{\La}^i(K_j,M) \ar[r]^{(f_j)^*} & \Ext_{\La}^i(K_{j-1},M)
\ar[r]^<<<{\cdot \eta_j} & \Ext_{\La}^{i+|\eta_j|}(K_{j-1},M) }$$
and
$$\xymatrix@C=15pt{
\Ext_{\La}^i(M,K_{j-1}) \ar[r]^{(f_j)_*} & \Ext_{\La}^i(M,K_j)
\ar[r] & \Ext_{\La}^{i-|\eta_j|+1}(M,K_{j-1}) \ar[r]^>>>>{\cdot
\eta_j} & \Ext_{\La}^{i+1}(M,K_{j-1}). }$$ From the upper exact
sequence, we see that the map $\Ext_{\La}^i(K_j,M)
\xrightarrow{(f_j)^*} \Ext_{\La}^i(K_{j-1},M)$ vanishes for $i \gg
0$, since the multiplication map involving $\eta_j$ is injective.
From the lower exact sequence we see that the map
$\Ext_{\La}^i(M,K_{j-1}) \xrightarrow{(f_j)_*} \Ext_{\La}^i(M,K_j)$
is surjective for $i \gg 0$.

The latter implies that when $i$ is large, the maps $f_1, \dots,
f_{c-1}$ induce a chain
$$\Ext_{\La}^i(M,M) \xrightarrow{(f_1)_*} \Ext_{\La}^i(M,K_1)
\xrightarrow{(f_2)_*} \cdots \xrightarrow{(f_{c-1})_*}
\Ext_{\La}^i(M,K_{c-1})$$ of epimorphisms. Now choose a homogeneous
element $\eta \in H^+$ which is regular on $\Ext_{\La}^i (K_{c-1},
\La / \ra )$ for $i \gg 0$. Applying $\phi_{K_{c-1}}$ to this
element gives an element
$$\phi_{K_{c-1}} ( \eta ) \colon 0 \to K_{c-1} \to K \to
\Omega_{\La}^{|\eta|-1}(K_{c-1}) \to 0$$ in
$\Ext_{\La}^*(K_{c-1},K_{c-1})$, where the module $K$ is projective.
Using the arguments in the proof of \cite[Corollary 3.2]{Bergh2}, we
see that $\phi_{K_{c-1}} ( \eta )$ cannot be nilpotent in
$\Ext_{\La}^*(K_{c-1},K_{c-1})$. Consequently, given any $w \in
\mathbb{N}$, there is an integer $i \ge w$ such that
$\Ext_{\La}^i(K_{c-1},K_{c-1})$ is nonzero. Using the exact
sequences $\phi_M ( \eta_1 ), \phi_{K_1} ( \eta_2 ), \dots,
\phi_{K_{c-2}}( \eta_{c-1})$ we then see that given any $w \in
\mathbb{N}$, there is an integer $i \ge w$ such that
$\Ext_{\La}^i(M,K_{c-1})$ is nonzero. This shows that the
composition
$$M \xrightarrow{f_1} K_1 \xrightarrow{f_2} \cdots
\xrightarrow{f_{c-1}} K_{c-1}$$ is nonzero in $\underline{\mod}
\Lambda$.

Now consider the functors $\underline{\Hom}_{\La}(K_j,-)$ on
$\underline{\mod} \Lambda$, together with the natural
transformations
$$\underline{\Hom}_{\La}(K_{c-1},-) \xrightarrow{(f_{c-1})^*}
\underline{\Hom}_{\La}(K_{c-2},-) \xrightarrow{(f_{c-2})^*} \cdots
\xrightarrow{(f_1)^*} \underline{\Hom}_{\La}(M,-).$$ Since the map
$\Ext_{\La}^i(K_j,M) \xrightarrow{(f_j)^*} \Ext_{\La}^i(K_{j-1},M)$
vanishes for $i \gg 0$, the map $\underline{\Hom}_{\La}(K_j,
\Omega_{\La}^i(M)) \xrightarrow{(f_j)^*}
\underline{\Hom}_{\La}(K_{j-1}, \Omega_{\La}^i(M))$ vanishes for $i
\ll 0$. From Lemma \ref{sequence} we conclude that for every module
$X \in \langle M \rangle_{c-1}$, the map
$$\underline{\Hom}_{\La}(K_{c-1}, \Omega_{\La}^i(X))
\xrightarrow{(f_{c-1} \circ \cdots \circ f_1)^*}
\underline{\Hom}_{\La}(M, \Omega_{\La}^i(X))$$ vanishes for $i \ll
0$. However, by \cite[Theorem 2.3]{Bergh1} the module $K_{c-1}$ is
periodic in $\underline{\mod} \Lambda$, that is, there is an integer
$p \ge 1$ such that $K_{c-1} \simeq \Omega_{\La}^p(K_{c-1})$ in
$\underline{\mod} \Lambda$. Therefore, since the composition
$f_{c-1} \circ \cdots \circ f_1$ is nonzero in $\underline{\mod}
\Lambda$, the map
$$\underline{\Hom}_{\La}(K_{c-1}, \Omega_{\La}^{ip}(K_{c-1}))
\xrightarrow{(f_{c-1} \circ \cdots \circ f_1)^*}
\underline{\Hom}_{\La}(M, \Omega_{\La}^{ip}(K_{c-1}))$$ does not
vanish for any $i \in \mathbb{Z}$. This shows that the module
$K_{c-1}$ cannot be an element in $X \in \langle M \rangle_{c-1}$,
and so $n \ge c$. The proof is complete.
\end{proof}

Using Proposition \ref{lowerbound} and Auslander's upper bound, we
obtain the promised result on the representation dimension. We
denote by $\ell \ell ( \La )$ the Loewy length of our algebra
$\La$.

\begin{theorem}\label{repdim}
If $\La$ is a non-semisimple selfinjective algebra and {\bf{Fg}}
holds, then
$$\cx \La / \ra +1 \le \repdim \La
\le \ell \ell ( \La ).$$
\end{theorem}

Rouquier showed that the representation dimension of the exterior
algebra on an $n$-dimensional vector space is exactly $n+1$. It
therefore seems natural to ask the following:

\begin{question}
When {\bf{Fg}} holds, what is the exact value of $\repdim \La$?
\end{question}

The following corollaries to Theorem \ref{repdim} provide lower
bounds for the representation dimension of the algebras given in the
three examples prior to Theorem \ref{dimension}. In particular, we
obtain \cite[Corollary 19]{Oppermann1}, half of \cite[Theorem
4.1]{Rouquier2} and the result of Avramov and Iyengar on the
representation dimension of Artin complete intersections (cf.\
\cite{Avramov2}).

\begin{corollary}\label{groups}
Suppose $k$ is a field of positive characteristic $p$, and let $G$
be a finite group whose order is divisible by $p$. Then $\repdim
kG \ge p- \rank G +1,$ that is, the representation dimension of
$kG$ is strictly greater than $\Krulldim \Ho^*(G,k)$.
\end{corollary}

\begin{proof}
By a result of Quillen (cf.\ \cite{Quillen1, Quillen2}), the
complexity of the trivial $kG$-module, i.e.\ the Krull dimension of
the cohomology ring $\Ho^*(G,k)$, equals the $p$-rank of $G$.
\end{proof}

\begin{corollary}\label{CI}
Let $A$ be a commutative Noetherian local complete intersection of
codimension $c$. If $A$ is Artin, then $\repdim A \ge c+1.$
\end{corollary}

\begin{proof}
By a classical result of Tate (cf.\ \cite[Theorem 6]{Tate}), the
complexity of the simple module over a complete intersection equals
the codimension of the ring.
\end{proof}

\begin{remark}
Let $A$ be a commutative Noetherian local complete intersection of
codimension $c$. If $A$ is complete, then the proof of Theorem
\ref{dimension} also applies to the stable category of finitely
generated maximal Cohen-Macaulay $A$-modules. Namely, the dimension
of this triangulated category is at least $c-1$.
\end{remark}

\begin{corollary}\label{Hochschild}
Suppose $\La$ is semisimple and projective as a $k$-module.
Furthermore, suppose the Hochschild cohomology ring $\HH^*( \La )$
is Noetherian, and that $\Ext_{\La}^*( \La / \ra, \La / \ra )$ is a
finitely generated $\HH^*( \La )$-module. Then $\repdim \La \ge
\Krulldim \HH^*( \La )+1.$ In particular, the representation
dimension of the exterior algebra on an $n$-dimensional vector space
is at least $n+1$
\end{corollary}

\begin{proof}
The Krull dimension of $\HH^* ( \La )$ is its rate of growth $\gamma
\left ( \HH^* ( \La ) \right )$ as a graded $k$-module. Therefore,
since the $\HH^* ( \La )$-module $\Ext_{\La}^* ( \La / \ra, \La /
\ra )$ is finitely generated, we see that
$$\cx \La / \ra = \gamma \left ( \Ext_{\La}^* ( \La / \ra, \La /
\ra ) \right ) \le \gamma \left ( \HH^* ( \La ) \right ) = \Krulldim
\HH^* ( \La ).$$ Denote the radical of $\La^{\e}$ by $\ra^{\e}$, and
let $B$ be a finitely generated bimodule (i.e.\ $B \in \mod
\La^{\e}$). If $B$ is not simple, then choose an exact sequence
$$0 \to S \to B \to B' \to 0$$
in which $S$ is simple. This sequence induces an exact sequence
$$\Ext_{\La^{\e}}^*( \La, S ) \to \Ext_{\La^{\e}}^*( \La, B ) \to
\Ext_{\La^{\e}}^*( \La, B' )$$ of $\HH^* ( \La )$-modules, all of
which are finitely generated by \cite[Proposition 2.4]{Erdmann}.
Consequently the inequality
$$\gamma \left ( \Ext_{\La^{\e}}^*( \La, B ) \right ) \le \max \{
\gamma \left ( \Ext_{\La^{\e}}^*( \La, S ) \right ), \gamma \left (
\Ext_{\La^{\e}}^*( \La, B' ) \right ) \}$$ holds, and so induction
on length gives
$$\gamma \left ( \Ext_{\La^{\e}}^*( \La, B ) \right ) \le
\gamma \left ( \Ext_{\La^{\e}}^*( \La, \La / \ra ) \right ) =
\cx_{\La^{\e}} \La.$$ In particular, the inequality $\gamma \left (
\HH^* ( \La ) \right ) \le \cx_{\La^{\e}} \La$ holds. But the
complexity of $\La$ as a bimodule equals that of the $\La$-module
$\La / \ra$. Namely, applying $- \otimes_{\La} \La / \ra$ to the
minimal projective bimodule resolution of $\La$ gives the minimal
$\La$-projective resolution of $\La / \ra$. Therefore
$\cx_{\La^{\e}} \La = \cx \La / \ra$, and this shows that the Krull
dimension of $\HH^* ( \La )$ equals $\cx \La / \ra$.
\end{proof}

We end with two examples illustrating Corollary \ref{Hochschild}.
These provide new examples of classes of algebras with arbitrarily
large representation dimension.

\begin{examples}
(i) Let $k$ be a field, let $n \ge 1$ be an integer, and let $\La$
be the quantum complete intersection
$$k \langle X_1, \dots, X_n \rangle / ( X_i^2, \{ X_iX_j -
q_{ij}X_jX_i \}_{i<j} ),$$ where $0 \neq q_{ij} \in k$.  This
algebra is finite dimensional of dimension $2^n$, and the complexity
of $k$ is $n$. Furthermore, this is a Frobenius algebra; the
codimension two argument in the beginning of \cite[Section
3]{BerghErdmann} carries over. In particular, this algebra is
selfinjective, and it was shown in \cite{ErdmannSolberg} that
{\bf{Fg}} holds if and only if all the $q_{ij}$ are roots of unity.
Therefore, when this is the case, then the representation dimension
of $\La$ is at least $n+1$.

(ii) Let $k$ be an algebraically closed field, and let $R$ be a
Noetherian Artin-Schelter regular Koszul $k$-algebra of dimension
$d$. That is, $R$ is graded connected of global dimension $d$, its
Gelfand-Kirillov dimension is finite, and
$$\Ext_R^i(k,R) \simeq \left \{
\begin{array}{ll}
0 & i \neq d \\
k & i=d \text{ (up to shift)}.
\end{array}
\right.$$ If $R$ is a finitely generated module over its center,
then by \cite[Proposition 9.15]{Solberg} the Koszul dual $\La$ of
$R$ is selfinjective, satisfies {\bf{Fg}}, and $\cx \La / \ra =d$.
Thus in this case the representation dimension of $\La$ is at least
$d+1$.

An example of such an algebra is obtained from the Sklyanin algebras
(cf.\ \cite[Section 8]{Smith}: let $E$ be an elliptic curve over
$k$, and fix a point $P \in E$ such that $nP=0$ for some $n \ge 1$.
Denote by $\sigma_P \colon E \to E$ the corresponding translation
automorphism. Furthermore, let $d \ge 1$ be an integer, and let
$A_d(E, \sigma_P )$ be the $d$-dimensional Sklyanin algebra. This is
an Artin-Schelter regular algebra of the above type.
\end{examples}


\begin{thebibliography}{EHSST}
\bibitem[Au1]{Auslander1}M.\ Auslander, \emph{Representation
dimension of Artin algebras}, Queen Mary College Mathematics
Notes, London, 1971, republished in \cite{Auslander2}.
\bibitem[Au2]{Auslander2}M.\ Auslander, \emph{Selected works of
Maurice Auslander. Part 1}, I.\ Reiten, S.\ Smal{\o}, {\O}.\
Solberg (editors), Amer.\ Math.\ Soc., Providence, 1999.
\bibitem[Avr]{Avramov1}L.\ Avramov, \emph{Modules of finite
virtual projective dimension}, Invent.\ Math.\ 96 (1989), 71-101.
\bibitem[AvI]{Avramov2}L.\ Avramov, S.\ Iyengar, \emph{The dimension
of the stable derived category of a complete intersection local
ring}, in preparation.
\bibitem[Be1]{Bergh1}P.A.\ Bergh, \emph{Complexity and periodicity},
Colloq.\ Math. 104 (2006), no.\ 2, 169-191.
\bibitem[Be2]{Bergh2}P.A.\ Bergh, \emph{Modules with reducible
complexity}, J.\ Algebra 310 (2007), 132-147.
\bibitem[BeE]{BerghErdmann}P.\ Bergh, K.\ Erdmann, \emph{(Co)homology of
quantum complete intersections}, preprint.
\bibitem[BeO]{BerghOppermann}P.\ Bergh, S.\ Oppermann, \emph{The
representation dimension of quantum complete intersections},
preprint.
\bibitem[EHSST]{Erdmann}K.\ Erdmann, M.\ Holloway, N.\ Snashall,
  {\O}.\ Solberg, R.\ Taillefer, \emph{Support varieties for
    selfinjective algebras}, K-theory 33 (2004), 67-87.
\bibitem[ErS]{ErdmannSolberg}K.\ Erdmann, {\O}.\ Solberg,
\emph{Finite generation of the Hochschild cohomology ring of some
    Koszul algebras}, in preparation.
\bibitem[Eve]{Evens}L.\ Evens, \emph{The cohomology ring of a
finite group}, Trans.\ Amer.\ Math.\ Soc.\ 101 (1961), 224-239.
\bibitem[IgT]{Igusa}K.\ Igusa, G.\ Todorov, \emph{On the finitistic
global dimension conjecture for Artin algebras}, in
\emph{Representations of algebras and related topics}, 201-204,
Fields Inst.\ Commun.\ 45, Amer.\ Math.\ Soc., Providence, 2005.
\bibitem[Iya]{Iyama}O.\ Iyama, \emph{Finitness of representation
dimension}, Proc.\ Amer.\ Math.\ Soc.\ 131 (2003), no.\ 4,
1011-1014.
\bibitem[KrK]{Krause}H.\ Krause, D.\ Kussin, \emph{Rouquier's
theorem on representation dimension}, in \emph{Trends in
representation theory of algebras and related topics}, 95-103,
Contemp.\ Math.\ 406, Amer.\ Math.\ Soc., Providence, 2006.
\bibitem[Op1]{Oppermann1}S.\ Oppermann, \emph{A lower bound for the
representation dimension of $kC_p^n$}, Math.\ Z.\ 256 (2007), no.\
3, 481-490.
\bibitem[Op2]{Oppermann2}S.\ Oppermann, \emph{Lower bounds for
Auslander's representation dimension}, preprint.
\bibitem[Qu1]{Quillen1}D.\ Quillen, \emph{The spectrum of an
equivariant cohomology ring, I}, Ann.\ of Math.\ 94 (1971), 549-572.
\bibitem[Qu2]{Quillen2}D.\ Quillen, \emph{The spectrum of an
equivariant cohomology ring, II}, Ann.\ of Math.\ 94 (1971),
573-602.
\bibitem[Ro1]{Rouquier1}R.\ Rouquier, \emph{Dimensions of
triangulated categories}, preprint.
\bibitem[Ro2]{Rouquier2}R.\ Rouquier, \emph{Representation dimension
of exterior algebras}, Invent. Math.\ 165 (2006), no.\ 2, 357-367.
\bibitem[Smi]{Smith}P.\ Smith, \emph{Some finite-dimensional
algebras related to elliptic curves}, in \emph{Representation theory
of algebras and related topics (Mexico City, 1994)}, 315-348, CMS
Conf.\ Proc.\ 19, Amer.\ Math.\ Soc., Providence, 1996.
\bibitem[Sol]{Solberg}{\O}.\ Solberg, \emph{Support varieties for
modules and complexes}, in \emph{Trends in representation theory of
algebras and related topics}, 239-270, Contemp.\  Math.\ 406, Amer.\
Math.\ Soc., Providence, 2006.
\bibitem[Tat]{Tate}J.\ Tate, \emph{Homology of noetherian rings and
local rings}, Ill.\ J.\ Math.\ 1 (1957), 14-27.
\bibitem[Yon]{Yoneda}N.\ Yoneda, \emph{Note on products in
$\Ext$}, Proc.\ Amer.\ Math.\ Soc.\ 9 (1958), 873-875.
\end{thebibliography}
\end{document}